\documentclass[12pt]{amsart}
\usepackage{amsmath,amsthm,amscd,amssymb,amsfonts, amsbsy}
\usepackage{latexsym}
\usepackage{txfonts}
\usepackage{enumerate}
\usepackage{verbatim}

\numberwithin{equation}{section}

\theoremstyle{plain}
\newtheorem{theorem}[equation]{Theorem}
\newtheorem{lemma}[equation]{Lemma}

\theoremstyle{definition}

\newtheorem{example}[equation]{Example}
\newtheorem*{acknowledgment}{Acknowledgment}

\theoremstyle{remark}
\newtheorem{remark}[equation]{Remark}

\newcommand{\dv}{\operatorname{div}}

\newcommand{\mysection}[1]{\section{#1}
\setcounter{equation}{0}}

\newcommand{\bR}{\mathbb R}
\newcommand{\bP}{\mathbb P}

\newcommand{\bZ}{\mathbb Z}

\newcommand\cL{\mathcal{L}}
\newcommand\cP{\mathcal{P}}

\newcommand\tT{\tilde{T}}
\newcommand\tD{\tilde{D}}
\newcommand\tv{\tilde{v}}
\newcommand\tu{\tilde{u}}
\newcommand\tf{\tilde{f}}
\newcommand\tvf{\tilde{\vec{f}}}

\newcommand\hT{\hat{T}}

\newcommand{\ip}[1]{\left\langle#1\right\rangle}

\newcommand{\abs}[1]{\left\lvert#1\right\rvert}

\renewcommand{\epsilon}{\varepsilon}
\renewcommand{\vec}[1]{\boldsymbol{#1}}

\begin{document}
\title[Partial Schauder estimates]
{Partial Schauder estimates for second-order elliptic and parabolic equations}

\author[H. Dong]{Hongjie Dong}
\address[H. Dong]{Division of Applied Mathematics, Brown University,
182 George Street, Providence, RI 02912, United States of America}
\email{Hongjie\_Dong@brown.edu}

\author[S. Kim]{Seick Kim}
\address[S. Kim]{Department of Computational Science and Engineering, Yonsei University, 262 Seongsanno, Seodaemun-gu, Seoul 120-749, Republic of Korea}

\email{kimseick@yonsei.ac.kr}



\begin{abstract}
We establish Schauder estimates for both divergence and non-divergence form second-order elliptic and parabolic equations involving H\"older semi-norms not with respect to all, but only with respect to some of the independent variables.

\end{abstract}

\maketitle

\mysection{Introduction}

The aim of this article is to obtain certain pointwise estimates, which we shall hereafter call \textit{partial Schauder estimates}, for both divergence and non-divergence form second-order elliptic and parabolic equations involving H\"older semi-norms not with respect to all, but only with respect to some of the independent variables.

To be more precise, let us first introduce some related notations.
Most notations are chosen to be compatible with those in \cite{Kr96}.
Let $x=(x^1,\ldots, x^d)$ be a point in $\bR^d$, with $d\ge 2$, and $q$ be an integer such that $1\le q <d$.
We distinguish the first $q$ coordinates of $x$ from the rest and write $x=(x',x'')$, where $x'=(x^1,\ldots,x^q)$ and $x''=(x^{q+1},\ldots,x^d)$.
For a function $u$ on $\bR^d$, we define a \textit{partial H\"older semi-norm} with respect to $x'$ as
\[
[u]_{x',\delta}:=\sup_{x''\in\bR^{d-q}}\,\sup_{\substack{x',y'\in \bR^q\\ x'\neq y'}}\frac {|u(x',x'')-u(y',x'')|}{|x'-y'|^\delta}.
\]
Throughout this article, we assume $0<\delta<1$ unless explicitly otherwise stated.
For  $k=0,1,2,\ldots$, we set
\[
[u]_{x',k+\delta}=[D_{x'}^k u]_{x',\delta}=\max_{\alpha\in\bZ_+^q,\, |\alpha|=k}[\tilde{D}^\alpha u]_{x',\delta},
\]
where we used the usual multi-index notation and $\tD^\alpha:=D_1^{\alpha_1}\cdot\ldots D_q^{\alpha_q}$.

Let $L$ be a uniformly elliptic operators in non-divergence form $Lu=a^{ij}D_{ij} u$, whose coefficients are measurable in $x$ and H\"older continuous in $x'$. Then a partial Schauder estimate for $L$ in the whole space $\bR^d$ is an estimate of the form
\begin{equation*}
\tag{\textasteriskcentered}\label{PS}
[u]_{x',2+\delta}\leq N [L u]_{x',\delta}+NK[D^2 u]_0,
\end{equation*}
where $N$ is a constant that depends only on $d$, $q$, $\delta$ and the ellipticity constant  of $L$, and $K$ is the partial H\"older semi-norm of the coefficients of $L$ with respect to $x'$; see Theorem~\ref{thm1} and Remark \ref{rem7} below for more precise statement.
Moreover, if the coefficients of the elliptic operator $L$ are constants, then we have a better partial Schauder estimate of the form
\begin{equation*}
\tag{\dag}\label{C}
[D_{x'} u]_{1+\delta}\leq N [Lu]_{x',\delta},
\end{equation*}
which means that if $Lu$ is H\"older continuous in $x'$, then $D_{xx'}u$ are H\"older continuous not just in $x'$ but in {\em all} variables. For the proof of \eqref{C}, we make use of the divergence structure in operators with constant coefficients.
We also give an example which shows the optimality of \eqref{C}.
It should be mentioned here that the estimate \eqref{C} is originally due to Fife \cite{Fife}, who actually treated elliptic equations of order $2m$ by means of the potential theory. However, our method also works for parabolic equation with coefficients merely {\em measurable} in the time variable, to which the potential theory is not applicable.
In this case, we prove that
\begin{equation*}
\tag{\ddag}\label{C2}
[D_{x'} u]_{(1+\delta)/2,1+\delta}\le N[Pu]_{x',\delta},
\end{equation*}
which implies that if $Pu:=u_t-a^{ij}(t)D_{ij}u$ is H\"older continuous in $x'$, then $D_{xx'} u$ are H\"older continuous in $(t,x)$; see Sect.~\ref{sec:c} for the details of the estimates \eqref{C}, \eqref{C2}, and other related results.

There is a vast literature on the classical ``full'' Schauder estimates of elliptic and parabolic equations. We refer readers to, for example, \cite{Brandt, Caffarelli89, Campanato, Knerr, KrPr, Lieb92, Lorenzi, Peetre, Safonov84, Safonov88, Simon, Trudinger, XJWang} and references therein.
Roughly speaking, the classical Schauder theory for second-order elliptic equations in non-divergence form says that if all the coefficients and data are H\"older continuous in all variables, then the same holds for the second derivatives of the solution.
The Schauder theory for second-order parabolic equations in non-divergence form says that if all the coefficients and data are H\"older continuous in the spatial variables and measurable in the time variable, then the same holds for the spatial second derivatives of the solution (see, e.g., \cite{Brandt, Knerr, Lieb92, Lorenzi}).\footnote{In many places, Schauder theory for parabolic equations may also refer to the result which says that if the coefficients and data are H\"older continuous in both space and time variables, then the same holds for the spatial second derivatives and the time derivative of the solution (see, e.g., \cite{Kr96, Lieberman}).}
These results were recently generalized in \cite{KrPr} to equations with growing lower order coefficients.

On the other hand, it seems to us that there is very little literature regarding Schauder estimates for elliptic and parabolic equations with coefficients and data that are regular only with respect to some of the independent variables.
We started investigating this problem after conversations with Professor Xu-Jia Wang, who recently informed us about a paper by Fife \cite{Fife} and an upcoming article by himself and Tian \cite{TW} on this subject. Another motivation of our paper is recent interesting work initiated by Krylov in \cite{Krylov_2005} on $L_p$-solvability of elliptic and parabolic equations with leading coefficients VMO in some of the independent variables.

Compared to previously known results, the novelty of our results is that, as we alluded earlier, we allow the coefficients of the operator to be very irregular in $x''$; the payoff is that our method only works for second-order elliptic and parabolic operators, where the maximum principle and Krylov-Safonov theory (or De Giorgi-Moser-Nash theory) are available. We also note that in the nondivergence case, the operators are allowed to be degenerate in $x''$; see Remark \ref{rem15}.

The organization of this paper is as follows. In Sect.~\ref{sec:m}, we state our main theorems and introduce some other notations.  The proofs of main theorems are given in Sect.~\ref{sec:e} and Sect.~\ref{sec:p}. Finally, we treat equations with coefficients independent of $x$ in Sect.~\ref{sec:c} and prove estimates \eqref{C} and \eqref{C2}.

\mysection{Main Results}		\label{sec:m}

First, we consider elliptic operators in non-divergence form
\begin{equation}
\label{eq0.1}
Lu:=a^{ij}(x'')D_{ij}u
\end{equation}
and elliptic operators in divergence form
\begin{equation}
\label{eq0.2}
\cL u:=D_i(a^{ij}(x'')D_j u),
\end{equation}
where the coefficients $a^{ij}(x)=a^{ij}(x'')$ are bounded measurable functions on $\bR^d$ that are independent of $x'$ and satisfy the uniform ellipticity condition
\begin{equation}
\label{elliptic}
\nu|\xi|^2\le a^{ij}(x)\xi^i\xi^j\le \nu^{-1}|\xi|^2,\quad\forall x\in\bR^{d},\,\, \xi\in \bR^d,
\end{equation}
for some constant $\nu\in (0,1]$.
We assume the symmetry of the coefficients (i.e., $a^{ij}=a^{ji}$) for the operators $L$ in non-divergence form but for the operators $\cL$ in divergence form, we instead assume that $\sum_{i,j=1}^d |a^{ij}|^2\leq\nu^{-2}$.

For  $k=0,1,2,\ldots$, we denote $C^{k}_{x'}(\bR^d)$ the set of all bounded measurable functions $u$ on $\bR^d$ whose derivatives $\tilde{D}^\alpha u$ for $\alpha\in \bZ_+^q$ with $|\alpha|\le k$ are continuous and bounded in $\bR^d$.
We denote by $C^{k+\delta}_{x'}(\bR^d)$ the set of all functions $u\in C^{k}_{x'}(\bR^d)$ for which the partial H\"older semi-norm $[u]_{x', k+\delta}$ is finite.
We use the notation $W^k_p(\bR^d)$, $k=1,2,\ldots$, for the Sobolev spaces in $\bR^d$.

We say that $u$ is a strong solution of $Lu=f$ in $\bR^d$ if $u\in W^2_{d,\,loc}(\bR^d)$ and satisfies the equation $Lu=f$ a.e. in $\bR^d$.
\begin{theorem}
                                    \label{thm1}
Let $u$ be a bounded strong solution of the equation
\[L u=f \quad \text{in }\,\bR^d,\]
where $f \in C^\delta_{x'}(\bR^d)$ and the coefficients $a^{ij}$ of the operator $L$ are continuous in $\bR^d$. Then $u\in C^{2+\delta}_{x'}(\bR^d)$ and there is a constant $N=N(d,q,\delta,\nu)$ such that
\begin{equation}
                                            \label{eq3.58}
[u]_{x',2+\delta}\le N[f]_{x',\delta}.
\end{equation}
\end{theorem}

\begin{remark}
In Theorem~\ref{thm1}, instead of assuming $u$ is a strong solution, we may assume that $u$ is a viscosity solution of $Lu=f$.
\end{remark}

\begin{remark}
The continuity assumption on the coefficients $a^{ij}$ is not essential in Theorem~\ref{thm1}, and the constant $N$ doesn't depend on the modulus of continuity of $a^{ij}$.
All that is needed for the proof is $W^2_d$-solvability of the Dirichlet problem \eqref{eq1.08}.
For example, we may assume that the coefficients $a^{ij}$ of $L$ belong to the class of VMO; see, e.g., \cite{CFL2}.
\end{remark}

We shall say that $u$ is a weak solution of $\cL u =\dv \vec f$ in $\bR^d$ if $u$ is a weak solution in $W^1_2(\Omega)$ of $\cL u =\dv \vec f$ for any bounded domain $\Omega\subset \bR^d$.

\begin{theorem}
                                    \label{thm3}
Let $u$ be a bounded weak solution of the equation
\[\cL u=\dv \vec f \quad \text{in }\,\bR^d,\]
where $\vec f=(f^1,\ldots,f^d)$ and $f^i\in C^{\delta}_{x'}(\bR^d)$ for $i=1,\ldots,d$.
Then $u\in C^{1+\delta}_{x'}(\bR^d)$ and there is a constant $N=N(d,q,\delta,\nu)$ such that
\[
[u]_{x',1+\delta}\le N[\vec{f}]_{x',\delta}.
\]
\end{theorem}

Next, we consider parabolic operators in non-divergence form
\begin{equation}
\label{eq0.3}
Pu:=u_t-a^{ij}(t,x'') D_{ij}u
\end{equation}
and parabolic operators in divergence form
\begin{equation}
\label{eq0.4}
\cP u:=u_t- D_i(a^{ij}(t,x'') D_j u),
\end{equation}
where $t\in \bR$ and $x=(x',x'')\in \bR^d$.
Here, we assume the coefficients $a^{ij}(t,x)=a^{ij}(t,x'')$ are bounded measurable functions on $\bR^{d+1}$ that are independent of $x'$ and satisfy the uniform parabolicity condition
\begin{equation}
\label{parabolic}
\nu|\xi|^2\le a^{ij}(t,x)\xi^i\xi^j\le \nu^{-1}|\xi|^2,\quad\forall (t,x)\in\bR^{d+1},\,\, \xi\in \bR^d,
\end{equation}
for some constant $\nu\in (0,1]$.
As in the elliptic case we assume the symmetry of the coefficients for the non-divergence form operators $P$ but for the operators $\cP$ in divergence form, we instead assume that $\sum_{i,j=1}^d |a^{ij}|^2\leq\nu^{-2}$.

For a function $u(t,x)=u(t,x',x'')$ on $\bR^{d+1}$, we define a partial H\"older semi-norm with respect to $x'$ as
\begin{equation}				\label{eq:fhs}
[u]_{x',\delta}:=\!\!\sup_{t\in \bR,\,x''\in\bR^{d-q}}\,\sup_{\substack{x', y'\in \bR^q\\ x'\neq y'}}\frac {|u(t,x',x'')-u(t,y',x'')|}{|x'-y'|^\delta}.
\end{equation}
Other related definitions such as $[u]_{x',k+\delta}$ ($k=0,1,2,\ldots$) are accordingly extended to functions $u=u(t,x)$ on $\bR^{d+1}$.
Let $Q$  be a domain in $\bR^{d+1}$. We say that $u\in W^{1,2}_p(Q)$ for some $p\ge 1$ if $u$ and its weak derivatives $D u$, $D^2 u$, and $u_t$ are in $L_p(Q)$.

We say that $u$ is a strong solution of $Pu=f$ in $\bR^{d+1}$ if $u \in W^{1,2}_{d+1,\,loc}(\bR^{d+1})$ and satisfies the equation $Pu=f$ a.e. in $\bR^{d+1}$.

\begin{theorem}
                                    \label{thm2}
Let $u$ be a bounded strong solution of the equation
\[P u=f \quad \text{in}\,\,\bR^{d+1},\]
where $f \in C^{\delta}_{x'}(\bR^{d+1})$ and the coefficients $a^{ij}$ of the operator $P$ are continuous in $\bR^{d+1}$. Then $u\in C_{x'}^{2+\delta}(\bR^{d+1})$ and there is a constant $N=N(d,q,\delta,\nu)$ such that
\[
[u]_{x',2+\delta}\le N[f]_{x',\delta}.
\]
\end{theorem}

We say that $u$ is a weak solution of $\cP u = \dv \vec f$ in $\bR^{d+1}$ if $u$ is a generalized solution from $V_2(Q)$ of $\cP u=\dv \vec f$ for any bounded cylinder $Q=(t_0,t_1)\times\Omega$ in $\bR^{d+1}$; see \cite[\S III.1]{LSU} for the definition of $V_2(Q)$, etc.

\begin{theorem}
                                    \label{thm4}
Let $u$ be a bounded weak solution of the equation
\[\cP u=\dv \vec f \quad \text{in}\,\,\bR^{d+1},\]
where $\vec f=(f^1,\ldots,f^d)$ and $f^i\in C^{\delta}_{x'}(\bR^{d+1})$ for $i=1,\ldots,d$.
Then  $u\in C^{1+\delta}_{x'}(\bR^{d+1})$ and there is a constant $N=N(d,q,\delta,\nu)$ such that
\[
[u]_{x',1+\delta}\le N[\vec{f}]_{x',\delta}.
\]
\end{theorem}

One may also wish to consider parabolic partial Schauder estimates regarding H\"older continuity in $t$ as well.
Let $z=(t,x)=(t,x^1,\ldots,x^d)$ be a point in $\bR^{d+1}$ and denote $z'=(t,x')=(t,x^1,\ldots,x^q)$.
We define the parabolic distance between the points $z_1'=(t_1,x_1')$ and $z_2'=(t_2,x_2')$ as
\[
\tilde{\rho}(z_1',z_2')=|x_1'-x_2'|+|t_1-t_2|^{1/2}.
\]
We define a partial H\"older semi-norm with respect to $z'$ as
\[
[u]_{z',\delta/2,\delta}:=\sup_{x''\in\bR^{d-q}}\,\sup_{\substack{z_i\in\bR^{q+1}\\z_1'\neq z_2'}}\frac {|u(z_1',x'')-u(z_2',x'')|}{\tilde{\rho}^\delta(z_1',z_2')}.
\]
By $C^{\delta/2,\delta}_{z'}(\bR^{d+1})$ we denote the set of all bounded measurable functions $u$ on $\bR^{d+1}$ for which $[u]_{z', \delta/2,\delta}<\infty$.
We also introduce $C^{1+\delta/2,2+\delta}_{z'}(\bR^{d+1})$ as the set of all bounded measurable functions $u$ for which  the derivatives $u_t$ and $\tD^\alpha u$ for $\alpha\in\bZ^q_+$ with $|\alpha|\le 2$ are continuous and bounded in $\bR^{d+1}$, and
\[
[u]_{z',1+\delta/2,2+\delta}:=[u_t]_{z',\delta/2,\delta}+[D^2_{x'}u]_{z',\delta/2,\delta}<\infty,
\]
where we used the notation
\[
[D_{x'}^k u]_{z',\delta/2,\delta}:=\max_{\alpha\in\bZ_+^q,\, |\alpha|=k}[\tilde{D}^\alpha u]_{z',\delta/2,\delta}, \quad k=0,1,2,\ldots.
\]
It is slightly more complicated to define $[u]_{z',(1+\delta)/2,1+\delta}$.
First, we define a semi-norm (see \cite[Chapter IV]{Lieberman})
\[
\ip{u}_{1+\delta}:=\sup_{x\in \bR^d}\, \sup_{\substack{t,s\in\bR\\ t\neq s}}\frac {|u(t,x)-u(s,x)|}{|t-s|^{(1+\delta)/2}}.
\]
Then we define
\[
[u]_{z',(1+\delta)/2,1+\delta}:=
[D_{x'} u]_{z',\delta/2,\delta}+\ip{u}_{1+\delta}.
\]
By $C^{(1+\delta)/2,1+\delta}_{z'}(\bR^{d+1})$ we denote the set of all bounded measurable functions $u$ for which the derivatives $\tD^\alpha u$ for $\alpha\in\bZ^q_+$ with $|\alpha|\le 1$ are continuous and bounded in $\bR^{d+1}$ and $[u]_{z',(1+\delta)/2,1+\delta}<\infty$.

If the coefficients $a^{ij}(t,x'')$ appearing in \eqref{eq0.3} and \eqref{eq0.4} are also independent of $t$ so that $a^{ij}=a^{ij}(x'')$, then we have the following theorems.

\begin{theorem}
                                    \label{thm5}
Let $u$ be a bounded strong solution of the equation
\[Pu=f \quad \text{in}\,\,\bR^{d+1},\]
where $f \in C^{\delta/2,\delta}_{z'}(\bR^{d+1})$ and the coefficients $a^{ij}$ of the operator $P$ are continuous in $\bR^{d+1}$ and independent of $z'$. Then $u\in C_{z'}^{1+\delta/2,2+\delta}(\bR^{d+1})$ and there is a constant $N=N(d,q,\delta,\nu)$ such that
\[
[u]_{z',1+\delta/2,2+\delta}\le N[f]_{z',\delta/2,\delta}.
\]
\end{theorem}

\begin{theorem}
                                    \label{thm6}
Let $u$ be a bounded weak solution of the equation
\[\cP u=\dv \vec f \quad \text{in}\,\,\bR^{d+1},\]
where $\vec f=(f^1,\ldots,f^d)$ and $f^i\in C_{z'}^{\delta/2,\delta}(\bR^{d+1})$ for $i=1,\ldots,d$ and the coefficients $a^{ij}$ of the operator $P$ are independent of $z'$.
Then  $u\in C^{(1+\delta)/2,1+\delta}_{z'}(\bR^{d+1})$ and there is a constant $N=N(d,q,\delta,\nu)$ such that
\[
[u]_{z',(1+\delta)/2,1+\delta}\le N[\vec{f}]_{z',\delta/2,\delta}.
\]
\end{theorem}

\begin{remark}
                                    \label{rem7}
Recall that the classical Schauder theory is built on the estimates of equations with constant coefficients by using a perturbation argument. In Theorem \ref{thm1}, the conditions of the coefficients $a^{ij}$ can be also relaxed to allow the dependence on $x'$.
For instance, we may assume that the coefficients $a^{ij}$ satisfy $[a^{ij}]_{x',\delta}\le K$ for some $K>0$, at the cost that $u$ should be assumed to have bounded derivatives up to second order and an additional term $N K [D^2 u]_0$ appears on the right-hand side of \eqref{eq3.58}.
See the remark at the end of the next section for the proof. All the other theorems stated above can be extended in a similar fashion as well.
\end{remark}

\begin{remark}
                                    \label{rem314}
An interesting related question is whether the partial Schauder estimates hold up to the boundary, say for equations in the half space with the zero Dirichlet condition on the boundary. In the special case that the normal direction is one of $x''$-directions, we can use the technique of odd extensions to get an equation in the whole space, and then deduce the regularity in the $x'$-directions. In general, the partial Schauder estimate does not hold up to the boundary even for the Laplace operator in the half space.
We have the following example in the half space $\{(x^1,x^2)\in\bR^2:x^1>0\}$, which is inspired by a similar example for parabolic equations recently suggested by M. V. Safonov to the authors. Let $u$ be a solution to the problem
\[
\Delta u=f:=\eta(x^1)\eta(x^2)\chi_{[0,\infty)}(x^2),\quad u(0,\cdot)=0,
\]
where $\eta$ is a smooth function on $\bR$ satisfying $\eta_1(t)=1$ for $|t|\leq 1$ and $\eta(t)=0$ for $|t|\ge 2$.
Notice that $v:=D_1^2 u$ satisfies $v(0,x^2)=\eta \chi_{[0,\infty)}(x^2)$ and $\Delta v=0$ in the strip $\{(x^1,x^2)\in \bR^2: 0<x^1<1\}$.
In particular, we have $v(0,-\epsilon)=0$ for any $\epsilon>0$.
On the other hand, it can be seen (e.g., via boundary Harnack's inequality) that for sufficiently small $\epsilon>0$ we have $v(\epsilon,-\epsilon)\ge \delta$ for some positive number $\delta$ independent of $\epsilon$.
So there is no control of the modulus of continuity of $D_1^2 u$ even if $f$ is smooth in $x^1$.
\end{remark}

\begin{remark}
                                \label{rem16}
Although in this paper we only focus on equations without lower order terms, it is worth noting that by observing the proofs below the theorems above can be extended to general linear elliptic and parabolic operators in nondivergence form
\begin{align*}
Lu&=a^{ij}(x'')D_{ij}u+b^i(x'') D_i u+c(x'')u,\\
Pu&=-u_t+a^{ij}(x'')D_{ij}u+b^i(x'') D_i u+c(x'')u,
\end{align*}
with bounded coefficients $b^i$ and $c$, and elliptic and parabolic operators in divergence form
\begin{align*}
\cL u&=D_i(a^{ij}(x'')D_j u+b^i(x'')u)+\hat b^j(x'')D_j u+c(x'') u,\\
\cP u&=-u_t+D_i(a^{ij}(x'')D_j u+b^i(x'')u)+\hat b^j(x'')D_j u+c(x'') u,
\end{align*}
with bounded coefficients $b^i,\hat b^i$ and $c$.
In these cases, an additional term $N|u|_0$ should appear on the right-hand side of the estimates.
\end{remark}

\mysection{The proofs: Elliptic estimates}		\label{sec:e}

We prove the main theorems in essence by following M. V. Safonov's idea of applying equivalent norms and representing solutions as sums of ``small'' and smooth functions.
However, his argument as reproduced in the proof of \cite[Theorem 3.4.1]{Kr96} is not directly applicable in our case by several technical reasons and to get around this difficulty we also make use of the mollification method of Trudinger \cite{Trudinger}.

For a function $v$ defined on $\bR^d$ and $\epsilon>0$, we define a \textit{partial mollification} of $v$ with respect to the first $q$ coordinates $x'$ as
\[
\tv^\epsilon(x',x''):=\frac{1}{\epsilon^q}\int_{\bR^q}v(y',x'') \zeta\left(\frac{x'-y'}{\epsilon}\right)\,dy'=
\int_{\bR^q}v(x'-\epsilon y',x'')\zeta(y')\,dy',
\]
where $\zeta(x^1,\ldots,x^q)=\prod_{i=1}^q \eta(x^i)$ and $\eta=\eta(t)$ is a smooth function on $\bR$ with a compact support in $(-1,1)$ satisfying $\int \eta=1$, $\int t \eta\,dt=0$, and $\int t^2 \eta \,dt=0$.
Then, by virtue of Taylor's formula, it is not hard to prove the following lemma for partial mollifications (see, e.g., \cite[Chapter 3]{Kr96}).

\begin{lemma}
                                    \label{lem12.25}
\begin{enumerate}[i)]
\item
Suppose $v\in C^{\delta}_{x'}(\bR^d)$. Then for any $\epsilon>0$,
\[
\epsilon^{1-\delta}\sup_{\bR^d}\, |D_{x'} \tv^\epsilon|+
\epsilon^{2-\delta}\sup_{\bR^d}\, |D^{2}_{x'} \tv^\epsilon|
 \le N(d,q,\delta,\eta) [v]_{x',\delta},
\]
\item
Suppose $v\in C_{x'}^{k+\delta}(\bR^d)\,\,(k=0,1,2)$. Then for any $\epsilon>0$,
\[
\sup_{\bR^d}|v-\tv^\epsilon| \le N(d,q,\delta,\eta)\epsilon^{k+\delta}[v]_{x',k+\delta}.
\]
\end{enumerate}
\end{lemma}

For $k=1,2,\ldots$ denote by $\tilde\bP_k$ the set of all functions $p=p(x',x'')$ on $\bR^d$ such that  $p(x',x'')$ is a polynomial of $x'\in \bR^q$ of degree at most $k$ for any $x''$.
We will also use the following notation for a \textit{partial Taylor's polynomial} of order $k$ with respect to $x'$ of a function $v$ at a point $x_0'$:
\[
\tT^k_{x_0'}v(x',x''):=\sum_{\alpha\in \bZ_+^q,\, |\alpha|\le k} \frac{1}{\alpha!} (x'-x_0')^\alpha
\tilde{D}^\alpha v(x_0',x'').
\]


\begin{proof}[Proof of Theorem \ref{thm1}]

First we derive an a priori estimate for $u$ assuming that $u\in C_{x'}^{2+\delta}(\bR^d)$.
Let $\kappa>2$ be a number to be chosen later.
Since $a^{ij}$ are independent of $x'$, we have for any $r>0$,
\[
L \tu^{\kappa r}=\tf^{\kappa r}.
\]
Let $x_0$ be a point in $\bR^d$ and for simplicity of notation, let us write $B_r=B_r(x_0)$.
Let $w\in W^2_{d,\,loc}(B_{\kappa r})\cap C^0(\overline B_{\kappa r})$ be a unique solution of the Dirichlet problem (see \cite[Corollary 9.18]{GT})
\begin{equation}							 \label{eq1.08}
\left\{
  \begin{aligned}
    L w = 0 \quad & \hbox{in $B_{\kappa r}$;} \\
    w=u-\tu^{\kappa r} \quad & \hbox{on $\partial B_{\kappa r}$.}
  \end{aligned}
\right.
\end{equation}
By the 
maximum principle and Lemma \ref{lem12.25} ii), we obtain
\begin{equation}
                                                        \label{eq1.27}
\sup_{B_{\kappa r}}|w|= \sup_{\partial B_{\kappa r}}|w|\le N(\kappa r)^{2+\delta}[u]_{x',2+\delta}.
\end{equation}
It follows from the theory of Krylov and Safonov that $w$ is locally H\"older continuous in $B_{\kappa r}$ with a H\"older exponent $\delta_0=\delta_0(d,\nu)\in (0,1)$.
Since $a^{ij}$ are independent of $x'$, it is reasonable to expect from \eqref{eq1.08} a better interior estimate for $w$ with respect to $x'$. Indeed, by using a technique of the finite difference quotients and bootstrapping (see, e.g., \cite[\S 5.3]{CaCa95}), one easily gets from the H\"older estimates of  Krylov and Safonov that, for any integer $\ell \ge 1$,
\begin{equation}
                                            \label{eq16.41}
[D^\ell_{x'}w]_{0;B_{\kappa r/2}}\le (\kappa r)^{-\ell} N(\ell,d,q,\nu) |w|_{0;B_{\kappa r}},
\end{equation}
where we used notation
\[
[D_{x'}^\ell w]_{0,B}=\max_{\alpha\in\bZ_+^q,\, |\alpha|=\ell} [\tilde{D}^\alpha w]_{0;B};\quad [w]_{0;B}=|w|_{0;B}=\sup_B |w|.
\]
In particular, with $\ell=3$, we get
\begin{align}
                                                    \label{Eq1.48}
|w-\tT^2_{x_0'}w|_{0;B_r} &\le N r^3 [D^3_{x'} w]_{0;B_r}\le Nr^3 [D^3_{x'} w]_{0;B_{\kappa r/2}}\\
\nonumber
&\le N \kappa^{-3} |w|_{0;B_{\kappa r}}\le N\kappa^{\delta-1} r^{2+\delta}[u]_{x',2+\delta},
\end{align}
where the last inequality is due to \eqref{eq1.27}.

On the other hand, it is clear that $v:=u-\tu^{\kappa r}-w$ satisfies
\begin{equation}
                                            \label{eq18.48}
\left\{
  \begin{aligned}
    L v = f-\tf^{\kappa r} \quad & \hbox{in $B_{\kappa r}$;} \\
    v=0 \quad & \hbox{on $\partial B_{\kappa r}$.}
  \end{aligned}
\right.
\end{equation}
Therefore, by the maximum principle and Lemma~\ref{lem12.25} ii) we have
\begin{equation}
                                                \label{eq2.14}
|u-\tu^{\kappa r}-w|_{0;B_{\kappa r}}\le N(\kappa r)^{2+\delta} [f]_{x',\delta}.
\end{equation}
By Lemma \ref{lem12.25} i), we also get
\begin{equation}
                                                \label{eq2.17}
|\tu^{\kappa r}-\tT^2_{x_0'} \tu^{\kappa r}|_{0;B_r} \le Nr^3 [D_{x'}^3 \tu^{\kappa r}]_{0;B_r}
\le N\kappa^{\delta-1} r^{2+\delta}[u]_{x',2+\delta}.
\end{equation}
Take $p=\tT^2_{x_0'}w+\tT^2_{x_0'} \tu^{\kappa r}\in \tilde\bP_2$. Then
combining \eqref{Eq1.48}, \eqref{eq2.14}, and \eqref{eq2.17} yields
\begin{align*}
|u-p|_{0;B_r}
&\le |u-\tu^{\kappa r}-w|_{0;B_r}+|\tu^{\kappa r}-\tT^2_{x_0'} \tu^{\kappa r}|_{0;B_r}+|w-\tT^2_{x_0'}w|_{0;B_r}\\
&\le N\kappa^{\delta-1} r^{2+\delta}[u]_{x',2+\delta}+N(\kappa r)^{2+\delta} [f]_{x',\delta}.
\end{align*}
This obviously implies
\begin{equation}
                                \label{eq2.28}
r^{-2-\delta}\,\inf_{p\in\tilde\bP_2}|u-p|_{0;B_r(x_0)}\le N\kappa^{\delta-1} [u]_{x',2+\delta}+N \kappa^{2+\delta} [f]_{x',\delta},
\end{equation}
for any $x_0\in \bR^d$ and $r>0$. We take the supremum of the left-hand side \eqref{eq2.28} with respect to $x_0\in \bR^d$ and $r>0$, and then apply \cite[Theorem 3.3.1]{Kr96} to get
\begin{equation}
                                \label{eq2.33}
[u]_{x',2+\delta}\le N\kappa^{\delta-1} [u]_{x',2+\delta}+N\kappa^{2+\delta} [f]_{x',\delta}.
\end{equation}
To finish the proof of \eqref{eq3.58} for $u\in C_{x'}^{2+\delta}(\bR^d)$, it suffices to choose a large $\kappa$ such that $N\kappa^{\delta-1}<1/2$.

Now we drop the assumption that $u\in C_{x'}^{2+\delta}(\bR^d)$ by another use of the partial mollification method. As noted earlier in the proof, since $a^{ij}$ are independent of $x'$, we have
\[
L \tilde u^{1/k}=\tilde f^{1/k},\qquad k=1,2,\ldots.
\]
Since $\tilde u^{1/k}\in C_{x'}^{2+\delta}(\bR^d)$, by the argument above, we have a uniform estimate
\[
[\tilde u^{1/k}]_{x',2+\delta}\le N[\tilde f^{1/k}]_{x',\delta}\le N[f]_{x',\delta},\quad k=1,2,\ldots.
\]
Moreover, $[\tilde u^{1/k}]_0\le [u]_0$ and $\tilde u^{1/k}$ converges locally uniformly to $u$ as $k$ tends to infinity. We thus conclude that  $u\in C_{x'}^{2+\delta}(\bR^d)$ and \eqref{eq3.58} holds. The theorem is proved.
\end{proof}

\begin{remark}
                                    \label{rem15}
In fact, the operator $L$ in Theorem \ref{thm1} is allowed to be degenerate in the $x''$ direction; i.e., the uniform ellipticity condition \eqref{elliptic} can be replaced by the following degenerate ellipticity condition\footnote{We would like to thank the referee for pointing this out to us.}
\[
\nu|\xi'|^2\le a^{ij}(x)\xi^i\xi^j \le \nu^{-1}|\xi|^2,\quad\forall x\in\bR^{d},\,\, \xi\in \bR^d,
\]
for some constant $\nu\in (0,1]$.
The reason is sketched as follows.
Denote by $B_r'$  the $q$-dimensional ball of radius $r$ centered at the origin. Let $w$ be the solution of
\[
\left\{
  \begin{aligned}
    L w = 0 \quad & \hbox{in $B_{\kappa r}'\times \bR^{d-q}$;} \\
    w=u-\tu^{\kappa r} \quad & \hbox{on $B_{\kappa r}'\times \bR^{d-q}$.}
  \end{aligned}
\right.
\]
Then $v:=u-\tu^{\kappa r}-w$ satisfies
\[
\left\{
  \begin{aligned}
    L v = f-\tf^{\kappa r} \quad & \hbox{in $B_{\kappa r}'\times \bR^{d-q}$;} \\
    v=0 \quad & \hbox{on $\partial B_{\kappa r}'\times \bR^{d-q}$,}
  \end{aligned}
\right.
\]
instead of \eqref{eq18.48}.
Notice that we still have the estimates \eqref{eq1.27} and \eqref{eq2.14}, but the Krylov-Safonov estimate is not available here since the equation is degenerate.
Instead, we prove \eqref{eq16.41} by using Bernstein's method; see, for instance,  \cite[Theorem 8.4.4]{Kr96}.
Let $\zeta\in C_c^\infty(B_1)$ be a cut-off function such that $\zeta=1$ on $B_{1/2}$. Denote $\zeta_{\kappa r}(x)=\zeta(x/\kappa r)$.
Consider the function
\[
W:=\zeta_{\kappa r}^2|D_{x'}w|^2+\mu(\kappa r)^{-2}|w|^2,
\]
where $\mu>0$ is a constant to be chosen later. Since $Lw=0$ and $L(D_{x'}w)=0$ in $B_{\kappa r}$, we have
\begin{align*}
LW&=2\mu(\kappa r)^{-2} a^{ij}D_i wD_j w+2(a^{ij}D_i\zeta_{\kappa r} D_j\zeta_{\kappa r}+\zeta_{\kappa r} L\zeta_{\kappa r})|D_{x'}w|^2\\
&\,\,+8a^{ij}\zeta_{\kappa r} D_i \zeta_{\kappa r} D_j D_{x'} w \cdot D_{x'}w+2\zeta_{\kappa r}^2 a^{ij}D_i D_{x'} u\cdot D_j D_{x'} u.
\end{align*}
By using Cauchy-Schwarz inequality
\[
\abs{a^{ij} \xi^i \eta^j} \le \sqrt{a^{ij}\xi^i \xi^j} \sqrt{a^{ij} \eta^i \eta^j}, \quad\forall \xi,\eta \in \bR^d,
\]
and Cauchy's inequality with $\epsilon$, we get
\begin{align*}
LW&\ge 2\mu(\kappa r)^{-2} a^{ij}D_i wD_j w+2(-3a^{ij}D_i\zeta_{\kappa r} D_j\zeta_{\kappa r}+\zeta_{\kappa r} L\zeta_{\kappa r})|D_{x'}w|^2\\
&\ge 2\left(\mu(\kappa r)^{-2}\nu-3a^{ij}D_i\zeta_{\kappa r} D_j\zeta_{\kappa r}+\zeta_{\kappa r} L\zeta_{\kappa r}\right)|D_{x'}w|^2\ge 0
\end{align*}
provided that $\mu$ is chosen sufficiently large. Therefore, by the maximum principle, we have
\[
|D_{x'}w|^2_{0;B_{\kappa r/2}}\le |W|_{0; B_{\kappa r}} \le |W|_{0;\partial B_{\kappa r}}\le N(\kappa r)^{-2}|w|^2_{0;B_{\kappa r}},
\]
which gives \eqref{eq16.41} for $\ell=1$. The general case can be deduced by an induction. The rest of the proof remains valid.
\end{remark}

\begin{proof}[Proof of Theorem \ref{thm3}]
As in the proof of Theorem~\ref{thm1}, let us first assume that  $u\in C^{1+\delta}_{x'}(\bR^d)$. Let $\kappa>2$ be a number to be chosen later.
Since $a^{ij}$ are independent of $x'$, we have for any $r>0$,
\begin{equation}
                                                    \label{eq3.05}
\cL \tu^{\kappa r}=\dv \tvf{}^{\kappa r}.
\end{equation}
Let $x_0$ be a point in $\bR^d$ and write $B_R=B_R(x_0)$.
Let $w\in W^1_2(B_{\kappa r})$ be a unique solution of the generalized Dirichlet problem (see \cite[Theorem~8.3]{GT})
\begin{equation}							 \label{eq3.08}
\left\{
  \begin{aligned}
    \cL w = 0 \quad & \hbox{in $B_{\kappa r}$;} \\
    w=u-\tu^{\kappa r} \quad & \hbox{on $\partial B_{\kappa r}$.}
  \end{aligned}
\right.
\end{equation}
By the weak maximum principle (see \cite[Theorem 8.1]{GT}) and Lemma \ref{lem12.25} ii), we obtain
\begin{equation}
                                                        \label{eq3.27}
\sup_{B_{\kappa r}}|w|= \sup_{\partial B_{\kappa r}} |w| \le N(\kappa r)^{1+\delta}[u]_{x',1+\delta}.
\end{equation}
It follows from the well-known De Giorgi-Moser-Nash theory that $w$ is locally H\"older continuous in $B_{\kappa r}$ with some exponent $\delta_0=\delta_0(d,\nu)\in (0,1)$. Again we use the technique of the finite difference quotients and bootstrapping to get that, for any integer $\ell \ge 1$,
\[
[D^\ell_{x'} w]_{0;B_{\kappa r/2}}\le N(\kappa r)^{-\ell} N(\ell,d,q,\nu) |w|_{0;B_{\kappa r}}.
\]
In particular, with $\ell=2$, we get
\begin{align}
                                                    \label{Eq3.48}
|w-\tT^1_{x_0'}w|_{0;B_r} &\le N r^2[D^2_{x'} w]_{0;B_r} \le Nr^2[D^2_{x'} w]_{0;B_{\kappa r/2}}\\
\nonumber
&\le N\kappa^{-2} |w|_{0;B_{\kappa r}}\le N\kappa^{\delta-1} r^{1+\delta}[u]_{x',1+\delta},\end{align}
where the last inequality is due to \eqref{eq3.27}.

On the other hand, $v:=u-\tu^{\kappa r}-w$ satisfies
\[
\left\{
  \begin{aligned}
    \cL v = \dv(\vec{f}-\tvf{}^{\kappa r}) \quad & \hbox{in $B_{\kappa r}$;} \\
    v=0 \quad & \hbox{on $\partial B_{\kappa r}$.}
  \end{aligned}
\right.
\]
By taking $v$ itself as a test function for the above equation, we get
\begin{equation}
\label{eq3.10v}
\|Dv\|_{L_2(B_{\kappa r})}
\le N\|\vec{f}-\tvf{}^{\kappa r}\|_{L_2(B_{\kappa r})}.
\end{equation}
To obtain an a priori bound for $v$, we first use a local boundedness estimate for the weak solution $v$ (see e.g., \cite[Theorem 8.17]{GT}) and get
\[
|v|_{0;B_{\kappa r/2}}\le N\kappa r |\vec{f}-\tvf{}^{\kappa r}|_{0;B_{\kappa r}}+N(\kappa r)^{-d/2}\|v\|_{L_2(B_{\kappa r})}.
\]
Then the Poincar\'e inequality (see e.g., \cite[(7.44)]{GT})
\[
\|v\|_{L_2(B_{\kappa r})}\le N\kappa r\|Dv\|_{L_2(B_{\kappa r})}
\]
together with \eqref{eq3.10v} and Lemma~\ref{lem12.25} ii) yields
\begin{equation}
                                                \label{eq3.14}
|v|_{0;B_{\kappa r/2}}\le N(\kappa r)^{1+\delta} [\vec{f}]_{x',\delta}.
\end{equation}
By Lemma~\ref{lem12.25} i), we also get
\begin{equation}
                                                \label{eq3.17}
|\tu^{\kappa r}-\tT^1_{x_0'} \tu^{\kappa r}|_{0;B_r}
\le Nr^2 [D_{x'}^2 \tu^{\kappa r}]_{0;B_r}
\le N\kappa^{\delta-1} r^{1+\delta}[u]_{x',1+\delta}.
\end{equation}
Take $p=\tT^1_{x_0'}w+\tT^1_{x_0'} \tu^{\kappa r}\in\tilde\bP_1$. Then
by \eqref{Eq3.48}, \eqref{eq3.14}, and \eqref{eq3.17}, we get
\begin{align*}
|u-p|_{0;B_r}
&\le |u-\tu^{\kappa r}-w|_{0;B_r}+|\tu^{\kappa r}-\tT^1_{x_0'} \tu^{\kappa r}|_{0;B_r}+|w-\tT^1_{x_0'}w|_{0;B_r}\\
&\le N\kappa^{\delta-1} r^{1+\delta}[u]_{x',1+\delta}+N(\kappa r)^{1+\delta} [\vec{f}]_{x',\delta}.
\end{align*}
The rest of proof is almost identical to that of Theorem~\ref{thm1} and omitted.
\end{proof}

\begin{remark}                                      \label{rem23.52}
We now give a proof of the claim made in Remark \ref{rem7}. Let
\[
L':=a^{ij}(x_0',x'')D_{ij}
\] and $w$ be the solution of \eqref{eq1.08} with $L'$ in place of $L$.
Let us also denote
\[
g=\big(a^{ij}(x_0',x'')-a^{ij}(x',x'')\big)D_{ij}u.
\]
Then, instead of \eqref{eq18.48}, $v$ satisfies the problem
\[
\left\{
  \begin{aligned}
   L'v=f-\tilde f^{\kappa r}+g-\tilde g^{\kappa r} \quad & \hbox{in $B_{\kappa r}$;} \\
    v=0 \quad & \hbox{on $\partial B_{\kappa r}$.}
  \end{aligned}
\right.
\]
Notice that we have (recall $\kappa >2$)
\[
|\tilde g^{\kappa r}|_{0;B_{\kappa r}}+|g|_{0;B_{\kappa r}}\le N K (\kappa r)^\delta [D^2 u]_{0}.
\]
Then similarly to \eqref{eq2.14}, the maximum principle yields
\[
|v|_{0;B_{\kappa r}}\le N(\kappa r)^{2+\delta} ([f]_{x',\delta}+K[D^2 u]_{0}).
\]
The rest of the proof is almost the same as that of Theorem~\ref{thm1}.
\end{remark}

\mysection{The proofs: Parabolic estimates}	\label{sec:p}

The proofs are similar to those in the previous section but some adjustments are needed.

\begin{proof}[Proofs of Theorem \ref{thm2} and \ref{thm4}]
Since we are dealing with partial H\"older semi-norms  with respect to $x'$ and not with respect to $t$,
the proofs of Theorems \ref{thm2} and \ref{thm4} are completely analogous to those of Theorems \ref{thm1} and \ref{thm3}.
We simply have to replace $B_r$ by $Q_r$, elliptic estimates by corresponding parabolic estimates, etc.
Since we will replicate very similar arguments in the proofs of Theorems \ref{thm5} and \ref{thm6} below, we omit the details here.
\end{proof}

We introduce a few more notation for the proofs of Theorems \ref{thm5} and \ref{thm6}.
For $z=(t,x)\in \bR^{d+1}$ let $Q_\rho(z)=(t-\rho^2,t)\times B_\rho(x)$ and $\partial_p Q_\rho(z)$ be its parabolic boundary.
We denote by $\hat\bP_1$ the set of all functions $p$ on $\bR^{d+1}$ of the form
\[
p(z)=p(t,x',x'')=\sum_{i=1}^q \alpha^i(x'') x^i + \beta(x''),
\]
and by $\hat\bP_2$ the set of all functions $p$ on $\bR^{d+1}$ of the form
\[
p(z)=p(t,x',x'')=\alpha(x'')t+\sum_{i=1}^q \alpha^i(x'') x^i+ \sum_{i,j=1}^q \alpha^{ij}(x'') x^i x^j + \beta(x'').
\]
Then we define the first-order partial Taylor's polynomial with respect to $z'=(t,x')$ of a function $v$ on $\bR^{d+1}$ at a point $z_0'=(t_0,x_0')$ as
\[
\hT^1_{z_0'}v(z',x''):=v(z_0',x'')+\sum_{i=1}^q  D_i v(z_0',x'') (x^i-x_0^i),
\]
and the second-order partial Taylor's polynomial of $v$ at $z_0'$ as
\begin{align*}
\hT^2_{z_0'}v(z',x'')&:=v(z_0',x'')+v_t(z_0',x'')(t-t_0)+\sum_{i=1}^q  D_i v(z_0',x'') (x^i-x_0^i)\\
&\quad+\frac{1}{2}\sum_{i,j=1}^q D_{ij} v(z_0',x'')(x^i-x_0^i)(x^j-x_0^j).
\end{align*}
Let $\zeta(z')=\zeta(t,x^1,\ldots,x^q)=\eta(t)\prod_{i=1}^q\eta(x^i)$, where $\eta$ is the same function as given in the previous section.
For $\epsilon>0$ let $\zeta_\epsilon(t,x')=\epsilon^{-q-2}\zeta(\epsilon^{-2}t,\epsilon^{-1}x')$ and define a partial mollification of $v$ with respect to $z'$ as
\begin{align*}
\hat v^\epsilon(t,x',x'')&=\int_{\bR^{q+1}}v(s,y',x'') \zeta_\epsilon(t-s,x'-y')\,ds dy'\\
&=\int_{\bR^{q+1}}v(t-\epsilon^2s,x'-\epsilon y',x'') \zeta(s,y')\,ds dy'
\end{align*}
The following lemma, the proof of which we also omit, is a parabolic analogue of Lemma~\ref{lem12.25}.
\begin{lemma}
                                    \label{lem2.25}
\begin{enumerate}[i)]
\item
Suppose $v\in C^{\delta/2,\delta}_{z'}(\bR^{d+1})$. Then for any $\epsilon>0$,
\[
\epsilon^{2-\delta}\sup_{\bR^{d+1}} |D_t \hat v^\epsilon|
+\epsilon^{2-\delta}\sup_{\bR^{d+1}} |D_{x'}^2 \hat v^\epsilon|
+\epsilon^{1-\delta} \sup_{\bR^{d+1}} |D_{x'} \hat v^\epsilon| \le N(d,q,\delta,\eta)[v]_{z',\delta/2,\delta}.
\]
\item
Suppose $v\in C_{z'}^{(k+\delta)/2,k+\delta}(\bR^{d+1})\,\,(k=0,1,2)$. Then for any $\epsilon>0$,
\[
\sup_{\bR^{d+1}}|v-\hat v^\epsilon| \le N(d,q,\delta,\eta)\epsilon^{k+\delta}[v]_{z',(k+\delta)/2,k+\delta}.
\]
\end{enumerate}
\end{lemma}

\begin{proof}[Proof of Theorem \ref{thm5}]
As in the proof of Theorem~\ref{thm1}, we may assume that $u\in C^{1+\delta/2,2+\delta}_{z'}(\bR^{d+1})$. Let $\kappa>2$ be a number to be chosen later.
Since $a^{ij}=a^{ij}(x'')$ are independent of $z'$, we have for any $r>0$,
\[
P \hat u^{\kappa r}=
\hat f^{\kappa r}.
\]
For $z_0\in\bR^{d+1}$ let us write $Q_\rho=Q_\rho(z_0)$.
Let $w\in W^{1,2}_{d+1,\,loc}(Q_{\kappa r})\cap C^0(\overline Q_{\kappa r})$ be a unique strong solution of the problem (see \cite[Theorem 7.17]{Lieberman})
\begin{equation}							 \label{eq1.08p}
\left\{
  \begin{aligned}
    P w = 0 \quad & \hbox{in $Q_{\kappa r}$;} \\
    w=u-\hat u^{\kappa r} \quad & \hbox{on $\partial_p Q_{\kappa r}$.}
  \end{aligned}
\right.
\end{equation}
By the maximum principle and Lemma \ref{lem2.25} ii), we obtain
\begin{equation}
                                                        \label{eq1.27p}
\sup_{Q_{\kappa r}}|w|= \sup_{\partial_p Q_{\kappa r}}|w|\le N(\kappa r)^{2+\delta}[u]_{z',1+\delta/2,2+\delta}.
\end{equation}
It follows from the Krylov-Safonov theory that $w\in C^{\delta_2/2,\delta_0}_{loc}(Q_{\kappa r})$ for some exponent $\delta_0=\delta_0(d,\nu)\in (0,1)$.
Since $a^{ij}$ are independent of $z'$, as in the proof of Theorem~\ref{thm1}, we have for any integers $\ell, m \ge 0$,
\begin{equation}
\label{eq2.76a}
[D^\ell_{x'} D^m_t w]_{0;Q_{\kappa r/2}}\le (\kappa r)^{-\ell-2m} N(\ell,m,d,q,\nu)
|w|_{0;Q_{\kappa r}},
\end{equation}
where we used notations
\[
[D_{x'}^\ell D^m_t w]_{0,Q}=\max_{\alpha\in\bZ_+^q,\, |\alpha|=\ell} [\tilde{D}^\alpha D^m_t w]_{0;Q};\quad [w]_{0;Q}=|w|_{0;Q}=\sup_Q |w|.
\]
Notice that Taylor's formula yields (see \cite[Theorem 8.6.1]{Kr96})
\begin{equation}
\label{eq2.60q}
|w-\hT^2_{z_0'}w|_{0;Q_r} \le N r^4 [D^2_t w]_{0;Q_r}+ N r^3 [D_{x'} D_t w]_{0;Q_r}+Nr^3 [D^3_{x'} w]_{0;Q_r}.
\end{equation}
Then we obtain from \eqref{eq2.60q}, \eqref{eq2.76a}, and \eqref{eq1.27p}
\begin{align}
                                                    \label{Eq1.48p}
|w-\hT^2_{z_0'}w|_{0;Q_r} &\le N\kappa^{-4} |w|_{0;Q_{\kappa r}}+ N\kappa^{-3} |w|_{0;Q_{\kappa r}}+ N\kappa^{-3} |w|_{0;Q_{\kappa r}}\\
\nonumber
&\le N \kappa^{-3}|w|_{0;Q_{\kappa r}}\le N\kappa^{\delta-1} r^{2+\delta}[u]_{z',1+\delta/2,2+\delta}.
\end{align}

On the other hand, $v:=u-\hat u^{\kappa r}-w$ satisfies
\[
\left\{
  \begin{aligned}
    Pv = f-\hat f^{\kappa r} \quad & \hbox{in $Q_{\kappa r}$;} \\
    v=0 \quad & \hbox{on $\partial_p Q_{\kappa r}$.}
  \end{aligned}
\right.
\]
Therefore, by the maximum principle and Lemma~\ref{lem12.25} ii) we have
\begin{equation}
                                                \label{eq2.14p}
|u-\hat u^{\kappa r}-w|_{0;Q_{\kappa r}}\le N(\kappa r)^{2+\delta} [f]_{z',\delta/2,\delta}.
\end{equation}
Then by \eqref{eq2.60q} and Lemma \ref{lem2.25} i), we get
\begin{align}
\label{eq2.17p}
|\hat u^{\kappa r}-\hT^2_{z_0'}\hat u^{\kappa r}|_{0;Q_r} &\le N \kappa^{\delta-2} r^{2+\delta} [u]_{z',1+\delta/2,2+\delta}+ N \kappa^{\delta-1}r^{2+\delta} [u]_{z',1+\delta/2,2+\delta}\\
\nonumber
&\le N \kappa^{\delta-1} r^{2+\delta} [u]_{z',1+\delta/2,2+\delta}.
\end{align}
Take $p=\hT^2_{z_0'}w+\hT^2_{z_0'} \hat u^{\kappa r}\in \hat\bP_2$. Then
combining \eqref{Eq1.48p}, \eqref{eq2.14p}, and \eqref{eq2.17p} yields
\begin{align*}
|u-p|_{0;Q_r}
&\le |u-\hat u^{\kappa r}-w|_{0;Q_r}+|\hat u^{\kappa r}-\hT^2_{z_0'} \hat u^{\kappa r}|_{0;Q_r}+|w-\hT^2_{z_0'}w|_{0;Q_r}\\
&\le N\kappa^{\delta-1} r^{2+\delta}[u]_{z',1+\delta/s,2+\delta}+N(\kappa r)^{2+\delta} [f]_{z',\delta/2,\delta}.
\end{align*}
Therefore, we have
\begin{equation}
                                \label{eq2.28p}
r^{-2-\delta}\,\inf_{p\in\hat\bP_2}|u-p|_{0;Q_r(z_0)}\le N\kappa^{\delta-1} [u]_{z',1+\delta/2,2+\delta}+N \kappa^{2+\delta} [f]_{z',\delta/2,\delta},
\end{equation}
for any $z_0\in \bR^{d+1}$ and $r>0$. By taking the supremum in \eqref{eq2.28p} and then applying \cite[Theorem 8.5.2]{Kr96}, we get
\begin{equation}
                                \label{eq2.33p}
[u]_{z',1+\delta/2,2+\delta}\le N\kappa^{\delta-1} [u]_{z',1+\delta/2,2+\delta}+N\kappa^{2+\delta} [f]_{z',\delta/2,\delta}.
\end{equation}
The rest of proof is repetitive and omitted.
\end{proof}

\begin{proof}[Proof of Theorem \ref{thm6}]
We proceed similarly as in the proof of Theorem~\ref{thm1} and assume that $u\in C^{(1+\delta)/2,1+\delta}_{z'}(\bR^{d+1})$. Let $\kappa>2$ be a number to be chosen later.
Since $a^{ij}=a^{ij}(x'')$ are independent of $z'$, we have for any $r>0$,
\[
\cP \hat u^{\kappa r}=\dv \hat{\vec{f}}{}^{\kappa r}.
\]
For $z_0\in\bR^{d+1}$ write $Q_\rho=Q_\rho(z_0)$.
Let $w\in V_2(Q_{\kappa r})$ be a generalized solution of the boundary value problem (see \cite[\S III.4]{LSU})
\begin{equation}							 \label{eq3.08p}
\left\{
  \begin{aligned}
    \cP w = 0 \quad & \hbox{in $Q_{\kappa r}$;} \\
    w=u-\hat u^{\kappa r} \quad & \hbox{on $\partial_p Q_{\kappa r}$.}
  \end{aligned}
\right.
\end{equation}
By the maximum principle (see \cite[\S III.7]{LSU}) and Lemma \ref{lem2.25} ii), we obtain
\begin{equation}
                                                        \label{eq3.27p}
\sup_{Q_{\kappa r}}|w|= \sup_{\partial_p Q_{\kappa r}} |w| \le N(\kappa r)^{1+\delta}[u]_{z',(1+\delta)/2,1+\delta}.
\end{equation}
By the De Giorgi-Moser-Nash theory we have $w\in C^{\delta_2/2,\delta_0}_{loc}(Q_{\kappa r})$ for some exponent $\delta_0=\delta_0(d,\nu)\in (0,1)$.
Using the assumption that $a^{ij}$ are independent of $z'$ and arguing as before, we obtain the interior estimate \eqref{eq2.76a}.
Then by \eqref{eq2.76a} and \eqref{eq3.27p}, we get
\begin{align}
                                                    \label{Eq3.48p}
|w-\hT^1_{z_0'}w|_{0;Q_r} &\le N r^2([D^2_{x'} w]_{0;Q_r}+[D_t w]_{0;Q_r})\\
&\le Nr^2([D^2_{x'} w]_{0;Q_{\kappa r/2}}+[D_t w]_{0;Q_{\kappa r/2}})\nonumber \\
\nonumber
&\le N\kappa^{-2} |w|_{0;Q_{\kappa r}}\le N\kappa^{\delta-1} r^{1+\delta} [u]_{z',(1+\delta)/2,1+\delta}.
\end{align}
Notice that $v:=u-\hat u^{\kappa r}-w$ is a generalized solution from $V_2(Q_{\kappa r})$ of the problem
\[
\left\{
  \begin{aligned}
    \cP v= \dv(\vec{f}-\hat{\vec f}{}^{\kappa r}) \quad & \hbox{in $Q_{\kappa r}$;} \\
    v=0 \quad & \hbox{on $\partial_p Q_{\kappa r}$.}
  \end{aligned}
\right.
\]
By taking $v$ itself as a test function for the above equation, we get
\begin{equation}
\label{eq3.10vp}
\|Dv\|_{L_2(Q_{\kappa r})}
\le N\|\vec{f}-\hat{\vec f}{}^{\kappa r}\|_{L_2(Q_{\kappa r})}.
\end{equation}
By a local boundedness estimate (see e.g., \cite[Theorem 6.17]{Lieberman}), we have
\[
|v|_{0;Q_{\kappa r/2}}\le N\kappa r |\vec{f}-\hat{\vec f}{}^{\kappa r}|_{0;Q_{\kappa r}}+N(\kappa r)^{-(d+2)/2}\|v\|_{L_2(Q_{\kappa r})}.
\]
Then the Poincar\'e inequality
\[
\|v\|_{L_2(Q_{\kappa r})}\le N\kappa r\|Dv\|_{L_2(Q_{\kappa r})}
\]
together with \eqref{eq3.10vp} and Lemma~\ref{lem2.25} ii) yields
\begin{equation}
                                                \label{eq3.14p}
|v|_{0;Q_{\kappa r/2}}\le N(\kappa r)^{1+\delta} [\vec{f}]_{z',\delta/2,\delta}.
\end{equation}
By Lemma~\ref{lem2.25} i), we also get
\begin{align}
                                                \label{eq3.17p}
|\hat u^{\kappa r}-\hT^1_{z_0'} \hat u^{\kappa r}|_{0;Q_r}
&\le Nr^2 ([D_{x'}^2 \hat u^{\kappa r}]_{0;Q_r}+[D_t \hat u^{\kappa r}]_{0;Q_r})\\
\nonumber
&\le N\kappa^{\delta-1} r^{1+\delta}[u]_{z',(1+\delta)/2,1+\delta}.
\end{align}
Take $p=\hT^1_{z_0'}w+\hT^1_{z_0'} \hat u^{\kappa r}\in\hat\bP_1$. Then
by \eqref{Eq3.48p}, \eqref{eq3.14p}, and \eqref{eq3.17p}, we get
\begin{align*}
|u-p|_{0;Q_r}
&\le |u-\hat u^{\kappa r}-w|_{0;Q_r}+|\hat u^{\kappa r}-\hT^1_{z_0'} \hat u^{\kappa r}|_{0;Q_r}+|w-\hT^1_{z_0'}w|_{0;Q_r}\\
&\le N\kappa^{\delta-1} r^{1+\delta}[u]_{z',(1+\delta)/2,1+\delta}+N(\kappa r)^{1+\delta} [\vec{f}]_{z',\delta/2,\delta}.
\end{align*}
Therefore, we have
\begin{equation}
                                \label{eq2.28p2}
r^{-1-\delta}\,\inf_{p\in\hat\bP_1}|u-p|_{0;Q_r(z_0)}\le N\kappa^{\delta-1} [u]_{z',(1+\delta)/2,1+\delta}+N \kappa^{1+\delta} [\vec f]_{z',\delta/2,\delta},
\end{equation}
for any $z_0\in \bR^{d+1}$ and $r>0$. By first taking the supremum in \eqref{eq2.28p2} and then using the equivalence of parabolic H\"older semi-norms similar to \cite[Theorem 8.5.2]{Kr96}, we obtain
\[
[u]_{z',(1+\delta)/2,1+\delta}\le N\kappa^{\delta-1} [u]_{z',(1+\delta)/2,1+\delta}+N\kappa^{2+\delta} [\vec f]_{z',\delta/2,\delta}.
\]
The rest of proof is repetitive and omitted.
\end{proof}

\mysection{Equations with coefficients independent of $x$}	\label{sec:c}

As pointed out in the introduction, if the coefficients of the elliptic operator $L$ are constants, then we have somewhat better partial Schauder estimates, namely \eqref{C}.
More precisely, we consider elliptic operators
\[
L_0u:=a^{ij}D_{ij}u,
\]
where $a^{ij}$ are constants satisfying the condition \eqref{elliptic}.
Then we have
\begin{theorem}
                        \label{thmA1}
Assume that $f \in C^\delta_{x'}(\bR^d)$ and $u$ is a bounded $W^2_{2,loc}$ solution of the equation
\[L_0 u=f \quad \text{in }\,\bR^d.\]
Then $D_{x'} u \in C^{1+\delta}(\bR^d)$ and there is a constant $N=N(d,q,\delta,\nu)$ such that
\begin{equation}
                                            \label{eqA10.47}
[D_{x'} u]_{1+\delta}\le N[f]_{x',\delta}.
\end{equation}
\end{theorem}

\begin{remark}
In the case when $q=d-1$, Theorem \ref{thmA1} implies an interesting result that the full Hessian $D^2 u\in C_{x'}^\delta(\bR^d)$ and
\[
[D^2 u]_{x',\delta}\le N[f]_{x',\delta}.
\]
\end{remark}

We give an example showing that Theorem \ref{thmA1} and thus Theorem \ref{thmA2} below are optimal in the sense that one cannot expect $D_{x''}^2 u\in C_{x'}^\delta$ if $q<d-1$.

\begin{example}
Recall the following well-known example in $\bR^2$:
\[u(x,y)=xy\{-\ln (x^2+y^2)\}^{1/2}\zeta(x,y),\]
where $\zeta$ is a smooth cut-off function in $\bR^2$ compactly supported on $B_1$ and equals to $1$ on $\bar B_{1/2}$. A direct calculation shows that $u_{xx}, u_{yy} \in C^0(\bR^2)$ but $u_{xy} \notin L^\infty(\bR^2)$.
If we set $v(x,y,z)= u(x,y)\sin(z)$, then we have
\[
\Delta v(x,y,z)=(u_{xx}+u_{yy}-u)(x,y)\sin(z),\quad v_{xy}(x,y,z)=u_{xy}(x,y)\sin(z),
\]
and thus $[\Delta v]_{z,\delta}<\infty$ but $[v_{xy}]_{z,\delta}=\infty$.
\end{example}

\begin{remark}
By using well-known properties of the fundamental solutions of elliptic equations with constant coefficients  and proceeding similarly as in the proof of Theorem~\ref{thm1}, we can extend Theorem \ref{thmA1} to higher order elliptic operators with constant coefficients (cf. \cite[Theorem 3.6.1]{Kr96}).
This would give an alternative proof of \cite[Theorem 3.1]{Fife}.
\end{remark}

Instead of proving Theorem~\ref{thmA1} directly, we will prove a parabolic version of it, which is new to the best of our knowledge.
We consider parabolic operators
\[
P_0u:=u_t-a^{ij}(t)D_{ij}u,
\]
where $a^{ij}(t)$ are functions depending only on $t$ in a measurable way and satisfying the condition \eqref{parabolic}.
In contrast to elliptic equations with constant coefficients, the potential theory is not applicable to this case.

\begin{theorem}
                                            \label{thmA2}
Assume that $f \in C^{\delta}_{x'}(\bR^{d+1})$ and $u$ is a bounded $W^{1,2}_{2,loc}$ solution of the equation
\[P_0 u=f \quad \text{in}\,\,\bR^{d+1}.\]
Then $D_{x'} u\in C^{(1+\delta)/2,1+\delta}(\bR^{d+1})$ and there is a constant $N=N(d,q,\delta,\nu)$ such that
\begin{equation}					\label{eq5.3pA}
[D_{x'} u]_{(1+\delta)/2,1+\delta}\le N[f]_{x',\delta}.
\end{equation}
\end{theorem}
\begin{proof}
We use the same strategy as in the earlier proofs, but with $u$ replaced by $D_j u$, where $j=1,\ldots,q$.
We may certainly assume that $u$ is infinitely differentiable in $x$ with bounded derivatives.
Let $\kappa>2$ be a number to be chosen later.
In this proof, we denote
\[
\bar u^\epsilon(t,x):=\frac{1}{\epsilon^d}\int_{\bR^d}u(t,y) \zeta\left(\frac{x-y}{\epsilon}\right)\,dy=
\int_{\bR^d}u(t,x-\epsilon y)\zeta(y)\,dy,
\]
where $\zeta(x)=\zeta(x^1,\ldots,x^d)=\prod_{i=1}^d \eta(x^i)$ and $\eta=\eta(t)$ is given as in Sect.~\ref{sec:e}.
Notice that we have for any $r>0$,
\[
P_0 \bar u^{\kappa r}=\bar f^{\kappa r}.
\]
For $z_0=(t_0,x_0)\in\bR^{d+1}$, let us write $Q_\rho=Q_\rho(z_0)$.
We regard $P_0$ as a divergence form operator, and for $j=1,\ldots,q$, let $w$ be a generalized solution from $V_2(Q_{\kappa r})$ of the problem
\begin{equation}							 \label{eq1.08pA}
\left\{
  \begin{aligned}
    P_0 w = 0 \quad & \hbox{in $Q_{\kappa r}$;} \\
    w=D_j u-D_j \bar u^{\kappa r} \quad & \hbox{on $\partial_p Q_{\kappa r}$.}
  \end{aligned}
\right.
\end{equation}
By the maximum principle and Lemma \ref{lem12.25} ii), we obtain
\begin{equation}
                                                        \label{eq1.27pA}
\sup_{Q_{\kappa r}}|w|= \sup_{\partial_p Q_{\kappa r}}|w|\le N(\kappa r)^{1+\delta}[D_j u]_{x,1+\delta},
\end{equation}
where $[D_j u]_{x,1+\delta}$ is defined similar to \eqref{eq:fhs} .

Analogous to \eqref{eq2.76a}, we have for any integer $\ell\ge 0$,
\begin{equation}
                                                            \label{eq2.76aA}
[D_x^\ell w]_{0;Q_{\kappa r/2}}\le (\kappa r)^{-\ell} N(\ell,d,\nu) |w|_{0;Q_{\kappa r}},
\end{equation}
We define the first-order Taylor's polynomial of $w$ with respect to $x$ at $x_0$ as
\[
\bar T^1_{x_0}w(t,x):=w(t,x_0)+\sum_{i=1}^d  D_i w(t,x_0) (x^i-x_0^i).
\]
Notice that by Taylor's formula,
\begin{equation}
                                                    \label{eq2.60qA}
\abs{w-\bar T^1_{x_0}w}_{0;Q_r} \leq N r^2 [D_x^2w]_{0;Q_r},
\end{equation}
This together with \eqref{eq2.76aA} and \eqref{eq1.27pA} yields
\begin{equation}
                                                    \label{Eq1.48pA}
\abs{w-\bar T^1_{x_0}w}_{0;Q_r} \le N\kappa^{\delta-1} r^{1+\delta}[D_j u]_{x,1+\delta}.
\end{equation}

On the other hand, $v:=D_j u-D_j \bar u^{\kappa r}-w$ satisfies
\[
\left\{
  \begin{aligned}
    P_0v =D_j(f-\bar f^{\kappa r}) \quad & \hbox{in $Q_{\kappa r}$;} \\
    v=0 \quad & \hbox{on $\partial_p Q_{\kappa r}$.}
  \end{aligned}
\right.
\]
Observe that
\begin{align*}
D_jf&=D_j\left(f(t,x)-f(t,x^1,\ldots,x^{j-1},x_0^j,x^{j+1},\ldots,x^d)\right),\\
D_j \bar f^{\kappa r}&=D_j\left(\bar f^{\kappa r}(t,x)-\bar f^{\kappa r}(t,x^1,\ldots,x^{j-1},x_0^j,x^{j+1},\ldots,x^d)\right).
\end{align*}
Then, by a similar argument that lead to \eqref{eq3.14p}, we obtain
\begin{equation}
                                                \label{eq2.14pA}
|v|_{0;Q_{\kappa r/2}}\leq N(\kappa r)^{1+\delta}\left([f]_{x',\delta}+[\bar f^{\kappa r}]_{x',\delta}\right) \leq
 N(\kappa r)^{1+\delta} [f]_{x',\delta}.
\end{equation}
Then by the estimate \eqref{eq2.60qA} applied to $D_j u^{\kappa r}$ and Lemma \ref{lem2.25} i), we get
\begin{equation}
                                                \label{eq2.17pA}
\abs{D_j \bar u^{\kappa r}-\bar T^1_{x_0}D_j \bar u^{\kappa r}}_{0;Q_r}\le N \kappa^{\delta-1}r^{1+\delta} [D_j u]_{x,1+\delta}.
\end{equation}
We denote by $\bar \bP_1$ the set of all functions $p$ on $\bR^{d+1}$ of the form
\[
p(z)=p(t,x)=\sum_{i=1}^d \alpha^i(t) x^i + \beta(t),
\]
By taking  $p=\bar T^1_{x_0}w+\bar T^1_{x_0} D_j u^{\kappa r}\in \bP_1$ and combining \eqref{eq2.14pA}, \eqref{eq2.17pA}, and \eqref{Eq1.48pA}, we obtain
\begin{align*}
\abs{D_j u-p}_{0;Q_r} &\leq \abs{D_j u-D_j u^{\kappa r}-w}_{0;Q_r}+\abs{D_j u^{\kappa r}-T^1_{x_0} D_j u^{\kappa r}}_{0;Q_r}+\abs{w-T^1_{x_0}w}_{0;Q_r}\\
&\le N\kappa^{\delta-1} r^{1+\delta}[D_j u]_{x,1+\delta}+N(\kappa r)^{1+\delta} [f]_{x',\delta}.
\end{align*}
Therefore, for $j=1,\ldots,q$, we have
\begin{equation}
                                \label{eq2.28pA}
r^{-1-\delta}\,\inf_{p\in \bar \bP_1}\abs{D_j u-p}_{0;Q_r(z_0)}\le N\kappa^{\delta-1} [D_j u]_{x,1+\delta}+N \kappa^{1+\delta} [f]_{x',\delta/2,\delta},
\end{equation}
for any $z_0\in \bR^{d+1}$ and $r>0$. By taking the supremum over $z_0\in \bR^{d+1}$ and $r>0$ in \eqref{eq2.28pA} and then applying \cite[Theorem 3.3.1]{Kr96}, we get
\begin{equation*}
[D_j u]_{x,1+\delta}\le N\kappa^{\delta-1} [D_j u]_{x,1+\delta}+N\kappa^{1+\delta} [f]_{x',\delta},
\end{equation*}
which implies by taking $\kappa$ sufficiently large as before
\begin{equation}
                                \label{eq2.33pA}
[D_{x'} u]_{x,1+\delta}\le N [f]_{x',\delta}.
\end{equation}
To derive H\"older continuity in $t$-variable, we again use the mollification method.
For any $z=(t,x)\in \bR^{d+1}$ and  $r>0$, by the triangle inequality,
\begin{multline*}
\abs{D_{x'}u(t+r^2,x)-D_{x'}u(t,x)} \leq \abs{D_{x'}u(t+r^2,x)-D_{x'} \bar u^r(t+r^2,x)}\\
+\abs{D_{x'} u(t,x)-D_{x'} \bar u^r(t,x)}+\abs{D_{x'}\bar u^r(t+r^2,x)-D_{x'}\bar u^r(t,x)}.
\end{multline*}
The first two terms on the right-hand side is bounded by $N r^{1+\delta}[f]_{x',\delta}$ due to Lemma \ref{lem12.25} ii) and \eqref{eq2.33pA}. To bound the last term, we write
\begin{align*}
D_{x'}\bar u^r(t+r^2,x)&-D_{x'}\bar u^r(t,x)=\int_0^{r^2}D_{x'}D_t \bar u^r(t+s,x)\,ds\\
&\qquad=\int_0^{r^2}D_{x'}\left(a^{ij}D_{ij} \bar u^r(t+s,x)+\bar f^r(t+s,x)\right)\,ds\\
&\qquad=\int_0^{r^2}\left(a^{ij}D_{ij} D_{x'} \bar u^r(t+s,x)+D_{x'}\bar f^r(t+s,x)\right)\,ds.
\end{align*}
By Lemma \ref{lem12.25} i) and \eqref{eq2.33pA}, we have
\begin{align*}
\abs{D_x^2 D_{x'} \bar u^r}_0 &\leq Nr^{\delta-1}[D_{x'} u]_{x,1+\delta}\le N r^{\delta-1}[f]_{x',\delta},\\
\abs{D_{x'} \bar f^r}_0 &\leq N r^{\delta-1}[f]_{x',\delta}.
\end{align*}
Combining the estimates above yields
\[
\abs{D_{x'}u(t+r^2,x)-D_{x'}u(t,x)} \le Nr^{\delta+1}[f]_{x',\delta},
\]
which implies
\begin{equation}				\label{eq5.16pA}
\ip{D_{x'}u}_{1+\delta}\le N[f]_{x',\delta}.
\end{equation}
Similarly, we have
\begin{multline*}
\abs{D_{xx'}u(t+r^2,x)-D_{xx'}u(t,x)} \le \abs{D_{xx'}u(t+r^2,x)-D_{xx'} \bar u^r(t+r^2,x)}\\
+\abs{D_{xx'}u(t,x)-D_{xx'} \bar u^r(t,x)}+\abs{D_{xx'} \bar u^r(t+r^2,x)-D_{xx'} \bar u^r(t,x)}.
\end{multline*}
Similar to above, the first two terms on the right-hand side is bounded by $N r^{\delta}[f]_{x',\delta}$ due to Lemma \ref{lem12.25} ii) and \eqref{eq2.33pA}.
To bound the last term, we write
\begin{align*}
D_{xx'}\bar u^r(t+r^2,x)&-D_{xx'} \bar u^r(t,x) =\int_0^{r^2}D_{xx'}D_t \bar u^r(t+s,x)\,ds\\
&\qquad=\int_0^{r^2}D_{xx'}\left(a^{ij}D_{ij} \bar u^r(t+s,x)+\bar f^r(t+s,x)\right)\,ds\\
&\qquad =\int_0^{r^2}\left(a^{ij}D_{ij} D_{xx'} \bar u^r(t+s,x)+ D_{xx'} \bar f^r(t+s,x)\right)\,ds.
\end{align*}
By Lemma \ref{lem12.25} i) and \eqref{eq2.33pA}, we have
\begin{align*}
\abs{D^2 D_{xx'} \bar u^r}_0 &\le Nr^{\delta-2}[D_{xx'} u]_{x,\delta}\le N r^{\delta-2}[f]_{x',\delta},\\
\abs{D_{xx'} \bar f^r}_0 
&\le N r^{\delta-2}[f]_{x',\delta}.
\end{align*}
Combining the estimates above yields
\[
\abs{D_{xx'}u(t+r^2,x)-D_{xx'}u(t,x)} \le Nr^{\delta}[f]_{x',\delta},
\]
which together with \eqref{eq2.33pA} implies
\begin{equation}				\label{eq5.17pA}
[D_{xx'}u]_{\delta/2,\delta}\le N[f]_{x',\delta}.
\end{equation}
By combining \eqref{eq5.16pA} and \eqref{eq5.17pA}, we obtain the desired estimate \eqref{eq5.3pA}.
The proof is complete.
\end{proof}

\begin{remark}
By the same reasoning as in Remark \ref{rem23.52}, Theorem \ref{thmA1} and \ref{thmA2} can be extended to operators with coefficients that are H\"older continuous with respect to $x'$.
\end{remark}

\begin{acknowledgment}
The authors are grateful to Xu-Jia Wang for bringing this problem to their attention, and Nicolai V. Krylov for helpful comments.
The authors also thank the referee for useful comments and suggestions.
Hongjie Dong was partially supported by the National Science Foundation under agreement No.~DMS-0800129.
Seick Kim was supported by Basic Science Research Program through the National Research Foundation of Korea(NRF) grant funded by the Korea government (MEST, No.~R01-2008-000-20010-0) and also by WCU(World Class University) program through the NRF funded by MEST (No.~R31-2008-000-10049-0).
\end{acknowledgment}


\end{document}